%% file: quantum-linear-groups.tex
\documentclass[12pt]{amsart}

\include{preamble}

\title{Homology of quantum linear groups}

\author{A. Kaygun}
\address{Istanbul Technical University, Istanbul, Turkey}
\email{kaygun@itu.edu.tr}

\author{S. Sütlü}
\address{Işık University, Istanbul, Turkey}
\email{serkan.sutlu@isikun.edu.tr}

\begin{document}

\begin{abstract}
  For every $n\geq 1$, we calculate the Hochschild homology of the quantum
  monoids $M_q(n)$, and the quantum groups $GL_q(n)$ and $SL_q(n)$ with
  coefficients in a 1-dimensional module coming from a modular pair in
  involution.
\end{abstract}

\maketitle
%\tableofcontents
\section*{Introduction}

In this article we calculate the Hochschild homology of the quantum monoids
$M_q(n)$, and quantum groups $GL_q(n)$ and $SL_q(n)$,
\cite{KlimSchm-book,ParshallWang:QuantumLinearGroups}, with coefficients in a
1-dimensional module ${}_{f_q^{-1}}k$ coming from a modular pair in involution
(MPI) defined on $GL_q(n)$ and $SL_q(n)$ when $q$ is not a roots of unity.  We have
an effective algorithm that calculates the explicit classes, and generates the
corresponding Betti sequences.  We also show that these homologies with
coefficients in the MPI are direct summands of the Hochschild homologies of
$GL_q(n)$ and $SL_q(n)$ with coefficients in themselves, albeit with a degree
shift.

To achieve our goal, we introduce a general Hochschild-Serre type spectral
sequence for flat algebra extensions of the form $Q\subseteq P$ through which
calculating the homology of $P$ reduces to calculating the homology of $Q$ and
the $Q$-relative homology of $P$.  We then calculate the homology of $M_q(n)$ by
using the lattice of extensions $M_q(a,b)\subseteq M_q(n)$ with
$1\leq a,b\leq n$, and related reductions in homology.  Since $GL_q(n)$ is the
localization of $M_q(n)$ at the quantum determinant, it is almost immediate to
obtain the homology of $GL_q(n)$ from that of $M_q(n)$ by a localization in
homology~\cite[1.1.17]{Loday-book}.  Calculating the homology of $SL_q(n)$
becomes immediate since $GL_q(n) = SL_q(n)\otimes k[\C{D}_q,\C{D}_q^{-1}]$ where
$\C{D}_q$ is the quantum determinant~\cite[Proposition]{LevaStaff93}.

The Hochschild and cyclic homology of $SL_q(2)$ are calculated
in~\cite{MasuNakaWata90,Ross90} using a specific resolution for $SL_q(2)$.  The
Hochschild homology of $SL_q(2)$ with coefficients in itself twisted by the
modular automorphism is calculated in~\cite{HadfKrae05}.  The Hochschild
cohomology of $SL_q(n)$ with coefficients in ${}_\eta k_\eta$
(Definition~\ref{trivial-character} and
Proposition~\ref{prop:quantum-matrices-with-trivial-coeff}) is computed
in~\cite{HadfKrae06} using the Koszul approach similar to \cite{Ross90}.
In~\cite{HadfKrae06} authors also show that Hochschild homology and cohomology
of $SL_q(n)$ satisfy Poincare duality.  The Hochschild homology of the
restricted dual $\C{O}_h(G)$ of the quantum universal enveloping algebra
$U_h(\G{g})$, on the other hand, is described in~\cite{FengTsyg91} for a
semi-simple Lie group $G$ with its Lie algebra $\G{g}$.  There are also results
for subgroups of $GL_q(n)$.  Recall that for a semi-simple Lie group $G$, the
kernel of the natural morphism $G_q\to G$ is called the \emph{quantized
  Frobenius kernel of $G$}.  When $q$ is a roots of unity with high enough
order, the cohomology of the quantized Frobenius kernel of the Borel subgroup
$B_q(n)$ of $GL_q(n)$ with trivial coefficients comes from the quantum group
$N_q(n)$ of the nilpotent part in even degrees, while it is trivial in odd
degrees by~\cite{Hu99}.  The Hochschild homology of quantized Frobenius kernels
of other groups are calculated in~\cite{GinzKuma93,ParsWang92,ParsWang93} at
roots of unity.

A smash biproduct algebra (or \emph{a twisted tensor product
  algebra~\cite{CapEtAl:TwistedTensorProduct}}) $A\#_R B$ involves a
two-way interaction $R\colon B\otimes A\to A\otimes B$ between unital
associative algebras $A$ and $B$, as opposed to a one-way-interaction
in an ordinary smash product of an algebra with a Hopf algebra.
Calculating the absolute homology of a smash biproduct reduces to
calculating the homology of a bisimplicial object as we observed in
\cite{KaygSutl18}.  In the same paper we calculated the homology of
several algebras including of $M_q(2)$ using extensions of the form
$A\subseteq A\#_R B$.  However, calculating homologies of $M_q(n)$ for
$n>2$ were beyond the reach of the methods we developed
in~\cite{KaygSutl18}.  We needed new tools and techniques to deal with
more general types of extensions as we observed
in~\cite{Kaygun:NCFibrations}.  Thus the method we develop in
Section~\ref{HS-Spectral-Sequence} in this paper evolved out
of~\cite{KaygSutl18} and \cite{Kaygun:NCFibrations} with a different
set of assumptions.

\subsection*{Plan of the paper}

After developing a Hochschild-Serre type spectral sequence for flat
algebra extensions $Q\subseteq P$ in
Section~\ref{HS-Spectral-Sequence}, we review some basic facts and
definitions about quantum monoids $M_q(n)$, and groups $GL_q(n)$ and
$SL_q(n)$ in Section~\ref{QLG}.  Then in Section~\ref{sect:MLN} we
calculate the Hochschild homology of $M_q(n)$, $GL_q(n)$ and $SL_q(n)$
for every $n\geq 1$ with coefficients in a 1-dimensional module
${}_{f_{q,n}^{-1}} k$ defined using a modular pair in involution
$(f_{q,n}^{-1},1)$ for the Hopf algebras $GL_q(n)$ and $SL_q(n)$. In
Section~\ref{sect:calculations}, we explicitly write our calculations
for the cases $n=2,3,4$.

\subsection*{Notations and conventions}

We fix a ground field $k$ and an element $q\in k^\times$ which is not
a roots of unity.  All unadorned tensor products $\otimes$ are over
$k$.  All algebras are assumed to be unital and associative, but not
necessarily commutative or finite dimensional over $k$.  We use
$\langle X\rangle$ to denote the two-sided ideal generated by a subset
$X$ of elements in an algebra, and $Span_k(X)$ for the $k$-vector
space spanned by a subset $X$ of elements in a vector space.  For a
fixed vector space $V$, we use $\Lambda^*(V)$ to denote the exterior
algebra over $V$, but used only as a vector space.  For a fixed
algebra $A$, we use $\CB_*(A)$ and $\CH_*(A)$ respectively for the bar
complex and the Hochschild complex of the algebra $A$.  We use
$H_*(A)$ to denote the Hochschild homology of $A$.  If the complexes,
and therefore, homology includes coefficients other than the ambient
algebra $A$, we will indicate this using a pair $(A,M)$.

\subsection*{Acknowledgments}

This work completed while the first author was at the Department of Mathematics
and Statistics of Queen's University on academic leave from Istanbul Technical
University.  The first author would like to thank both organizations for their
support.

\section{A Hochschild-Serre Type Spectral Sequence}\label{HS-Spectral-Sequence}

In this section, we assume $P$ and $Q$ are unital associative
algebras, together with a fixed morphism of algebras $\vp:Q\to P$.

\subsection{The mapping cylinder algebra}\label{subsect-mapping-cylinder-alg}

Let $Z := P\oplus Q$ be the unital associative algebra given by the product
\[ 
(p,q)\cdot(p',q') = (pp' + p\varphi(q') + \varphi(q)p', qq') 
\] 
for any $(p,q),(p',q')\in Z$, and the unit $(0,1) \in Z$. Accordingly,
$P \leq Z$ is an ideal.  Moreover, since $P$ is a unital algebra, the homology
of its bar complex vanishes in any positive degree. As such, $P \leq Z$ is an
$H$-unital ideal,~\cite{Wodz89}.

Let $X$ be a left $P$-module, and $Y$ a right $P$-module. Then, they both can
be considered as $Z$-modules via
\[ 
x\lt(p,q) := x\lt p + x\lt\varphi(q)  \qquad (p,q)\rt y = p\rt y + \varphi(q)\rt y. 
\]
% Indeed, 
% \begin{align*}
%   (x\triangleleft(p,q))\triangleleft(p',q')
%    = & (x\triangleleft p + x\triangleleft\varphi(q))\triangleleft(p',q')\\
%    = & (x\triangleleft p) \triangleleft p' 
%      + (x\triangleleft p) \triangleleft \varphi(q')
%      + (x\triangleleft \varphi(q)) \triangleleft p' 
%      + (x\triangleleft \varphi(q)) \triangleleft \varphi(q')\\
%    = & x\triangleleft pp' 
%      + x\triangleleft p \varphi(q')
%      + x\triangleleft \varphi(q) p' 
%      + x\triangleleft \varphi(qq')\\
%    = & x\triangleleft (pp' + p\varphi(q') + \varphi(q)p' + qq') = x \lt ((p,q)\cdot(p',q'))
% \end{align*}
for any $x \in X$, and any $(p,q),(p',q') \in Z$. Similarly, $Y$ is a left $Z$-module.

\subsection{A filtration on the bar complex}

Let us now consider the bar complex $\CB_*(X,Z,Y)$ of the mapping cylinder algebra $Z:=P\oplus Q$ with coefficients in a right $P$-module $X$ and a left $P$-module $Y$.

We consider first the increasing filtration on $\CB_*(X,Z,Y)$ given by
\[ 
  G^i_{i+j} = \sum_{n_0+\cdots+n_i= j} X\otimes P^{\otimes
    n_0}\otimes Z\otimes\cdots\otimes P^{\otimes n_{i-1}}\otimes Z
  \otimes P^{\otimes n_i}\otimes Y \subseteq \CB_{i+j}(X,Z,Y).
\]
The bar differential interacts with the filtration as
$d(G^i_{i+j}) \subseteq G^{i}_{i+j-1}$.
Then the associated graded complex is given by
\begin{equation}\label{CB-d_0}
 E^0_{i,j} := G^{i}_{i+j}/G^{i-1}_{i+j} \cong 
   \bigoplus_{n_0+\cdots+n_i=j}
   X\otimes P^{\otimes n_0}\otimes Q\otimes\cdots\otimes
   P^{\otimes n_{i-1}}\otimes Q \otimes P^{\otimes n_i}\otimes Y,  
\end{equation}
with induced differentials.  However, since we cannot reduce the
number of $Q$'s in the quotient complex, one can think of the quotient
complex is a graded product of bar complexes with induced differentials
\[
  E^0_{i,*} = \CB_*(X,P,Q)\otimes_Q \underbrace{\CB_*(Q,P,Q) \otimes_Q\cdots\otimes_Q\CB_*(Q,P,Q)}_{i\text{-times}}\otimes_Q \CB_*(Q,P,Y)
\]
where $P$ acts on $Q$ via 0.  Since we observe that $P$ acts on $Q$
trivially, the $E^1$-term is given by
\[
E^1_{i,j} = H_{i+j}(E^0_{i,\ast}; d_0) = \begin{cases}
\tor_j^P(X,Y) & \text{\rm if} \,\, i=0,\\
0 & \text{\rm if}\,\, i\neq 0.
\end{cases}
\]
The spectral sequence then degenerates (for further details on spectral sequences we refer the reader to \cite{McCleary:SpectralSequences}), and we arrive at the following result.

\begin{proposition}\label{prop:HomologyCylinderAlgebra}
Given two algebras $P$ and $Q$, together with an algebra morphism $\vp:Q\to P$, let $Z:=P\oplus Q$ be the mapping cylinder algebra, $X$ a right $P$-module, and $Y$ be a left $P$-module. Then, $\tor_n^Z(X,Y) \cong \tor_n^P(X,Y)$ for all $n\geq 0$.
\end{proposition}

\subsection{A filtration on the Hochschild homology}\label{remark-Hochschild-version-I}

It is possible to adapt this setting to the Hochschild homology complex. To this end, given a $P$-bimodule $M$, we start with the increasing filtration
\[ 
G^{i}_{i+j} =  \sum_{n_0+\cdots+n_i= j}
 M \ot P^{\otimes n_0}\otimes Z\otimes\cdots\otimes
P^{\otimes n_{i-1}}\otimes Z \otimes P^{\otimes n_i} \subseteq \CH_{i+j}(M,Z)
\]
of $\CH_\ast(Z,M)$. It follows from the observation $b(G^i_{i+j}) \subseteq G^i_{i+j-1}$ that the Hochschild complex $\CH_\ast(Z,M)$ is a filtered differential complex. The associated graded complex is then
\[ E^0_{i,j} := G^i_{i+j}/G^{i-1}_{i+j} = 
   \bigoplus_{n_0+\cdots+n_i=j}
   M \ot P^{\otimes n_0}\otimes Q\otimes\cdots\otimes
   P^{\otimes n_{i-1}}\otimes Q \otimes P^{\otimes n_i} 
\]
together with the induced differential $b_0:E^0_{i,j} \lra E^0_{i,j-1}$ given similar to that of \eqref{CB-d_0} where $P$ acts on $Q$ by 0. Hence, the first page of the associated differential complex appears to be
\[
E^1_{i,j} = H_{i+j}(E^0_{i,\ast}; b_0) = \begin{cases}
H_j(P,M) & \text{\rm if} \,\, i=0,\\
0 & \text{\rm if}\,\, i\neq 0.
\end{cases}
\]
The spectral sequence then degenerates, and we arrive at the following result. 

\begin{proposition}\label{prop:HochschildCylinderAlgebra}
Given two algebras $P$ and $Q$, together with an algebra morphism $\vp:Q\to P$, let $Z:=P\oplus Q$ be the mapping cylinder algebra, $X$ a right $P$-module, and $Y$ be a left $P$-module. Then, $H_n(Z,M) \cong H_n(P,M)$ for any $n \geq 0$, and any $P$-bimodule $M$.   
\end{proposition}

\subsection{A second filtration on the bar complex}

Let $P$ and $Q$ be two algebras as above, but this time we assume $\vp:Q \to P$ is a left (or right) flat algebra morphism. In other words, $P$ is flat as a left (resp. right) $Q$-module via $\vp:Q \to P$. Let us now consider the increasing filtration
\[ 
  F^i_{i+j} =  \sum_{n_0+\cdots+n_i= j}
  X\otimes Q^{\otimes n_0}\otimes Z\otimes\cdots\otimes
  Q^{\otimes n_{i-1}}\otimes Z \otimes Q^{\otimes n_i}\otimes Y\subseteq \CB_{i+j}(X,Z,Y)
\] 
on the bar complex $\CB_\ast(X,Z,Y)$. Since $d(F^i_{i+j}) \subseteq F^i_{i+j-1}$, the bar complex $\CB_\ast(X,Z,Y)$ becomes a filtered differential complex; whose associated differential graded complex is
\[ 
E^0_{i,j} = F^{i}_{i+j}/F^{i-1}_{i+j} = \bigoplus_{n_0+\cdots+n_i=j}
   X\otimes Q^{\otimes n_0}\otimes P\otimes\cdots\otimes
   Q^{\otimes n_{i-1}}\otimes P \otimes Q^{\otimes n_i}\otimes Y,
\]
together with the induced differentials $d_0:E^0_{i,j} \lra E^0_{i,j-1}$ coming from the bar complex $\CB_*(X,Z,Y)$.  As in the case of our first filtration, one can view the resulting complex as a graded multi-product of bar complexes
\begin{equation}
  \label{CB-d_0-II}
  \CB_*(X,Q,P)\otimes_P \underbrace{\CB_*(P,Q,P)\otimes_P \cdots\otimes_P \CB_*(P,Q,P)}_{i\text{-times}} \otimes_Q \CB_*(P,Q,Y).
\end{equation}
% \begin{align}\label{CB-d_0-II}
%   d_0(x \ot & q_{\ell_1} \odots q_{\ell_{n_0}} \ot p_1 \odots p_i \ot q_{k_1} \odots q_{k_{n_i}} \ot y)\nonumber\\
%   = & x \cdot q_{\ell_1} \ot q_{\ell_2} \odots q_{\ell_{n_0}} \ot p_1 \odots p_i \ot q_{k_1} \odots q_{k_{n_i}} \ot y \\
%     & + \sum_{s=1}^{n_0-1} \,(-1)^s\, x \ot q_{\ell_1} \odots q_{\ell_s}q_{\ell_{s+1}}\odots q_{\ell_{n_0}} \ot p_1 \odots p_i \ot q_{k_1} \odots q_{k_{n_i}} \ot y \\ 
%     & + (-1)^{n_0}\,x \ot q_{\ell_1} \odots q_{\ell_{n_0-1}}  \ot q_{\ell_{n_0}} \cdot p_1 \odots p_i \ot q_{k_1} \odots q_{k_{n_i}} \ot y + \ldots + \\
% & (-1)^{n_0 + \ldots + n_{i-1} + 1}\, x \ot q_{\ell_1} \odots q_{\ell_{n_0}} \ot p_1 \odots p_i\cdot q_{k_1} \ot q_{k_2} \odots q_{k_{n_i}} \ot y + \\
% & \sum_{s=n_0+\ldots + n_{i-1} + 2}^{n_0 + \ldots + n_i-1} \,(-1)^s\, x \ot q_{\ell_1} \odots q_{\ell_{n_0}} \ot p_1 \odots p_i \ot q_{k_1} \odots q_{k_s}q_{k_{s+1}}\odots q_{k_{n_i}} \ot y + \\
% & (-1)^{n_0 + \ldots + n_i}\, x \ot q_{\ell_1} \odots q_{\ell_{n_0}} \ot p_1 \odots p_i \ot q_{k_1} \odots q_{k_{n_i-1}}\ot q_{k_{n_i}} \cdot y.
% \end{align}
In view of the assumption (that $P$ is flat as a $Q$-module), the $E^1$-term is given by
\[
  E^1_{i,j} = H_{i+j}(E^0_{i,\ast}; d_0)
  = \begin{cases}
    \tor^Q_j(X \ot_Q \underbrace{P\otimes_Q \ldots \ot_Q P}_{i\text{-times}},Y) & \text{if $P$ is flat as a left $Q$-module},\\
    \tor^Q_j(X, \underbrace{P\otimes_Q \ldots \ot_Q P}_{i\text{-times}}\otimes_Q Y)  & \text{if $P$ is flat as a right $Q$-module}.
  \end{cases}
\]
Keeping in mind that this spectral sequence converges to the $\tor$-groups of the mapping cylinder algebra $Z$, which are in turn identified with the $\tor$-groups of the algebra $P$ in the previous subsection, we obtain the result which may be summarized in the following proposition.

\begin{proposition}\label{prop-E1}
Given two algebras $P$ and $Q$, together with the left (resp. right) flat algebra morphism $\vp:Q\to P$, let $X$ a right $P$-module and $Y$ be a left $P$-module. Then, there is a spectral sequence such that
\begin{align*}
& E^1_{i,j} = \tor^Q_j(X \ot_Q \underbrace{P\otimes_Q \ldots \ot_Q P}_{i\text{-times}},Y) \Rightarrow \tor_{i+j}^P(X,Y), \\
& \hspace{2cm} \Big(\text{ resp. }\, E^1_{i,j} = \tor^Q_j(X, \underbrace{P\otimes_Q \ldots \ot_Q P}_{i\text{-times}} \otimes_Q Y) \Rightarrow \tor_{i+j}^P(X,Y) \Big).
\end{align*}
\end{proposition}

\subsection{A second filtration on the Hochschild homology}\label{Subsect-Hochschild-E1-E2}

We can present the arguments of the previous subsection in terms of the Hochschild homology as well.

To this end, we begin with the increasing filtration
\[ 
F^i_{j+i} =  \sum_{n_0+\cdots+n_i= j}
  M \ot Q^{\otimes n_0}\otimes Z\otimes\cdots\otimes
  Q^{\otimes n_{i-1}}\otimes Z \otimes Q^{\otimes n_i},
  \] 
  that satisfies $b(F_{i+j}^i) \subseteq F_{i+j-1}^i$. Then the associated graded complex is
  \[ 
  E^0_{i,j} = F^{i}_{j+i}/F^{i-1}_{j+i} = \bigoplus_{n_0+\cdots+n_i=j}
  M \ot   Q^{\otimes n_0}\otimes P\otimes\cdots\otimes
  Q^{\otimes n_{i-1}}\otimes P \otimes Q^{\otimes n_i}.
  \]
  Passing to the homology with respect to $b_0:E^0_{i,j} \lra E^0_{i,j-1}$, an analogue of \eqref{CB-d_0-II}, we arrive at the first page of the spectral sequence which is given by 
  \[
E^1_{i,j} = H_{i+j}(E^0_{i,\ast}; b_0) = H_j\Big(Q, M\ot_Q \underbrace{P\ot_Q \cdots \ot_Q P}_{i\text{-times}}\Big) 
  \]
that converges to  $H_\ast(Z,M)$ which is isomorphic to $H_\ast(P,M)$ by Proposition~\ref{prop:HochschildCylinderAlgebra}. Thus, we have proved the following proposition.
 
\begin{proposition}\label{prop:HochschildSS}
 Given two algebras $P$ and $Q$, together with the left (resp. right) flat algebra morphism $\vp:Q\to P$, let $M$ be a $P$-bimodule.  Then there is a spectral sequence whose first page is given by
  \[ E^1_{i,j} = H_j(Q, M\ot_Q \underbrace{P\ot_Q \cdots \ot_Q P}_{i\text{-times}}) \]
  that converges to the Hochschild homology $H_\ast(P,M)$.
\end{proposition}

\section{Quantum Linear Groups}\label{QLG}

\subsection{The algebra of quantum matrices $M_q(n,m)$}

Let $n$ and $m$ be two positive integers.  Following
\cite[Lemma 2.10]{HangYacobi:QuantumPolynomialFunctors}, we define $M_q(n,m)$ as
the associative algebra on $nm$ generators $x_{ij}$ where
$1\leq i\leq n$ and $1\leq j\leq m$.  These generators are subject to
the following relations
\begin{align}
  x_{j \ell } x_{i \ell } = & q \,x_{i \ell } x_{j \ell}
  \quad\text{ for all } 1\leq i < j\leq n\text{ and } 1\leq \ell \leq m,\label{quant-mat-I}\\
  x_{\ell j}x_{\ell i}  = & q\,x_{\ell i}x_{\ell j}
  \quad\text{ for all } 1\leq i < j\leq m\text{ and } 1\leq \ell\leq n,\label{quant-mat-II}\\
  x_{\ell i}x_{kj} = & x_{kj}x_{\ell i}
  \quad\text{ for all } 1\leq k < \ell \leq n\text{ and } 1\leq i<j\leq m,\label{quant-mat-III}\\
  x_{ki}x_{\ell j} - x_{\ell j}x_{ki} = & (q^{-1}-q)\, x_{kj}x_{\ell i}
    \quad\text{ for all } 1\leq k < \ell \leq n\text{ and } 1\leq i<j\leq m.\label{quant-mat-IV}
\end{align}
For convenience, we are going to use $M_q(n)$ for $M_q(n,n)$.  We also
use the following convention: for $a\leq n$ and $b\leq m$ when we
write $M_q(a,b)\subseteq M_q(n,m)$ we mean that we use the subalgebra
generated by $x_{ij}$ for $1\leq i\leq a$ and $1\leq j\leq b$ in
$M_q(n,m)$.  Notice that these generators are subject to the same
relations, and therefore, the canonical map $M_q(a,b)\to M_q(n,m)$ is
injective.

It follows from \eqref{quant-mat-I} and \eqref{quant-mat-II} that all column or row subalgebras
\begin{equation}
  \Cl_\ell :=\langle x_{i\ell}\mid 1 \leq i \leq n\rangle \qquad
  \Rw_\ell :=\langle x_{\ell j}\mid 1 \leq j \leq n \rangle
\end{equation}
are isomorphic to the quantum affine $n$-space $k_q^n$ which is defined as the $k$-algebra
\begin{equation}
  \label{eq:quantum-affine-space}
   M_q(1,n) \cong M_q(n,1) \cong k_q^n := k\{x_1, \ldots, x_n\}/\langle x_jx_i-q\,x_jx_i\mid i<j\rangle.
\end{equation}
See \cite[Subsect. 3.1]{GuccGucc97} for the multiparametric version.
Next, we note from \cite[Thm. 3.5.1]{ParshallWang:QuantumLinearGroups}
and \cite[Prop. 9.2.6]{KlimSchm-book} that
\[
\C{B} = \left \{\prod_{1 \leq i,j \leq n}x_{ij}^{t_{ij}}\mid t_{ij} \geq 0\right\}
\]
is a vector space basis of $M_q(n)$, with respect to any fixed order of the generators. 

\subsection{The bialgebra structure on $M_q(n)$}

The algebra $M_q(n)$ of quantum matrices is a bialgebra whose comultiplication $\D\colon M_q(n) \to M_q(n) \ot M_q(n)$ is given by
\[ \D(x_{ij}) := \sum_k x_{ik} \ot x_{kj} \]
and whose counit $\ve: M_q(n) \to k$ is given by
\[ \ve(x_{ij}) = \d_{ij}. \]

\subsection{The quantum determinant}

Let $S_n$ be the group of permutations of the set $\{1,\ldots,n\}$, and let $\ell(\s) \in \B{N}$ be the length of $\s \in S_n$. Let also $I:=\{i_1, \ldots, i_m\}$ and $J:=\{j_1, \ldots, j_m\}$ be two subsets of $\{1,\ldots,n\}$ such that $i_1 < \ldots < i_m$ and $j_1 < \ldots < j_m$. Then, the element 
\[
\C{D}_{IJ} := \sum_{\sigma\in S_m}\,(-q)^{\ell(\sigma)}\, x_{i_{\sigma(1)}j_1} \ldots x_{i_{\sigma(m)}j_m}  = \sum_{\sigma\in S_m}\,(-q)^{\ell(\sigma)}\, x_{i_1j_{\sigma(1)}} \ldots x_{i_mj_{\sigma(m)}} \in M_q(m) 
\]
is called the \emph{quantum m-minor determinant} as defined in~\cite[Sect. 9.2.2]{KlimSchm-book} and \cite[Sect. 4.1]{ParshallWang:QuantumLinearGroups}.

On one extreme we have $\C{D}_{IJ} = x_{ij}$  for $I=\{i\}$ and $J=\{j\}$.  On the other extreme, if we let $I = J = \{1,\ldots,n\}$ we get the \emph{quantum determinant} which is denoted by $\C{D}_q$. % It is concluded in \cite[Prop. 9.7]{KlimSchm-book} and \cite[Lemma 4.1.1]{ParshallWang:QuantumLinearGroups} that 
% \[
% \C{D}_q = \sum_{\sigma\in S_n}\,(-q)^{\ell(\sigma)-\ell(\mu)}\, x_{\sigma(1)\mu(1)} \ldots x_{\sigma(n)\mu(n)}  = \sum_{\sigma\in S_n}\,(-q)^{\ell(\sigma)-\ell(\mu)}\, x_{\mu(1)\sigma(1)} \ldots x_{\mu(n)\sigma(n)}
% \]
% for any $\mu \in S_n$.
The quantum determinant is in the center $M_q(n)$. Moreover, if $q$ is not a root of unity, then the center of $M_q(n)$ is generated by the quantum determinant.  For this result see \cite[Prop. 9.9]{KlimSchm-book}, \cite[Thm. 4.6.1]{ParshallWang:QuantumLinearGroups}, or \cite{NoumYamaMima93}.

\subsection{The quantum general linear group $GL_q(n)$}

The quantum group $GL_q(n)$ is obtained by adjoining $\C{D}_q^{-1}$ to the bialgebra $M_q(n)$. More precisely,
\[
GL_q(n) = \frac{M_q(n)[t]}{\langle t\C{D}_q - 1\rangle}.
\]
Let us note from this definition that $M_q(n)$ is a subalgebra of $GL_q(n)$. In terms of generators and relations, $GL_q(n)$ is the algebra generated by $n^2 + 1$ generators $x_{ij}$ and $t$ with $ i,j\in\{1,\ldots, n\}$, satisfying the same relations as \eqref{quant-mat-I} - \eqref{quant-mat-IV}, and 
\begin{align}
& \C{D}_qt = t\C{D}_q = 1, \label{GLq-relations-I} \\
& x_{ij}t = t x_{ij}. \label{GLq-relations-II}
\end{align}
On the other hand, $GL_q(n)$ is the localization $M_q(n)_{\C{D}_q}$ of $M_q(n)$ with respect to $\C{D}_q$ as in \cite[Sect. 5.3]{ParshallWang:QuantumLinearGroups}, and \cite[Prop. 1.1.17]{Loday-book}. As such, the bialgebra structure on $M_q(n)$ extends uniquely to $GL_q(n)$~\cite[Lemma 5.3.1]{ParshallWang:QuantumLinearGroups}. Furthermore, $GL_q(n)$ is a Hopf algebra with the antipode $S\colon GL_q(n) \to GL_q(n)$ given by
\[
 S(x_{ij}) := (-q)^{j-i}A_{ji}\C{D}_q^{-1}, \quad S(\C{D}_q^{-1}) : = \C{D}_q,
\]
where $A_{ij} := \C{D}_{IJ}$ with $I= \{1,\ldots,n\}-\{i\}$, and $J= \{1,\ldots,n\}-\{j\}$. The matrix 
\[
\left[q^{i-j}A_{ij} \right]_{1\leq i,j\leq n} 
\]
is called the \emph{quantum cofactor matrix} of $\left[x_{ij}\right]_{1\leq i,j \leq n}$, and the Hopf algebra $GL_q(n)$ is called the \emph{quantum general linear group}.

\subsection{The quantum special linear group $SL_q(n)$}

Next, we recall briefly the quantum version of the special linear group. It is given as the quotient space
\[
SL_q(n):= \frac{GL_q(n)}{\langle \C{D}_q-1\rangle} = \frac{M_q(n)}{\langle \C{D}_q-1\rangle},
\]
which happens to be a Hopf algebra with the bialgebra structure induced from $GL_q(n)$, or from $M_q(n)$, and the antipode $S\colon SL_q(n) \to SL_q(n)$ is given by
\[ S(x_{ij}) := q^{j-i}A_{ji} \]
induced from $GL_q(n)$. The Hopf algebra $SL_q(n)$ is called the \emph{quantum special linear group}.

Actually, one can write $GL_q(n)$ as a direct product of $SL_q(n)$ and
the Laurent polynomial ring over the quantum determinant
$k[\C{D}_q,\C{D}_q^{-1}]$.
\begin{proposition}{\cite[Proposition]{LevaStaff93}}\label{LS:SLNq}
  There is an isomorphism of algebras of the form
  \[ GL_q(n) \cong SL_q(n)\otimes k[\C{D}_q,\C{D}_q^{-1}]. \]
\end{proposition}

\subsection{Modular pairs of involution for $GL_q(n)$ and $SL_q(n)$}

Finally, we are going to see that both Hopf algebras $GL_q(n)$ and $SL_q(n)$
admit a \emph{modular pair in involution} (MPI).  Let us recall
from~\cite{ConnMosc99} that an Hopf algebra $H$ is said to admit an MPI if there
is an algebra homomorphism $\d\colon H \to k$ and a group-like element
$\s \in H$ such that
\[
  \d(\s) = 1 \text{ and } S_\d^2(h) = \s h \s^{-1}
  \text{ where } S_\d(h):= \d(h_{(1)})S(h_{(2)})
\]
for any $h \in H$.

\begin{proposition}\label{MPI-GL-SL}
  Let $f_{q,n}^{-1}:GL_q(n)\to k$ (resp. $f_{q,n}^{-1}\colon SL_q(n)\to k$) be
  given by $f_{q,n}^{-1}(x_{ij}) := \d_{ij}q^{(n+1)-2i}$. Then,
  $(f_{q,n}^{-1},1)$ is a MPI for the Hopf algebra $GL_q(n)$, (resp. for
  the Hopf algebra $SL_q(n)$.)
\end{proposition}

\begin{proof}
  We will give the proof for $GL_q(n)$.  The proof for the case of
  $SL_q(n)$ is similar, and therefore, is omitted. It is given in
  \cite[Lemma 5.4.1]{ParshallWang:QuantumLinearGroups} that $f_{q,n}\colon GL_q(n) \to k$ given by 
  \[ f_{q,n}(x_{ij}) = \d_{ij}q^{2i-(n+1)} \] is an algebra homomorphism,
  i.e. a character.  Then, its convolution inverse $f^{-1}_{q,n}\colon GL_q(n) \to k$ is also a character.  The claim then
  follows from the observation that
  \[ \widetilde{S}^2_{f_{q,n}^{-1}} = f_{q,n} \ast S^2 \ast f_{q,n}^{-1} \] since
  $S^2 = f_{q,n}^{-1} \ast \Id \ast f_{q,n}$ by
  \cite[Thm. 5.4.2]{ParshallWang:QuantumLinearGroups} where $\ast$ is
  the convolution multiplication on the set of characters of Hopf
  algebras.
\end{proof}

As for $GL_q(n)$, there is a second choice of MPI.
\begin{proposition}\label{MPI-GL}
  Let $f_{q,n}^{-1}:GL_q(n)\to k$ be as before, and let
  $\C{D}_q^{-1} \in GL_q(n)$ be the quantum determinant. Then,
  $(f_{q,n}^{-1},\C{D}_q^{-1} )$ is a modular pair in involution for the Hopf
  algebra $GL_q(n)$.
\end{proposition}

\begin{proof}
We have, for any $x\in GL_q(n)$,
\begin{align*}
  \widetilde{S}^2_{f_{q,n}^{-1}}(x)
  = & \widetilde{S}_{f_{q,n}^{-1}}\left(\C{D}_q^{-1}f_{q,n}^{-1}(x_{(2)})S(x_{(1)})\right) = \widetilde{S}_{f_{q,n}^{-1}}(S(x_{(1)}))f_{q,n}^{-1}(x_{(2)}) S(\C{D}_q^{-1}) \\
  = & \C{D}_q^{-1}f_{q,n}^{-1}(S(x_{(1)}))S^2(x_{(2)})f_{q,n}^{-1}(x_{(3)})\C{D}_q \\
  = & \C{D}_q^{-1}x\C{D}_q = x.
\end{align*}
Furthermore, we have
\begin{align}\label{f(D)}
  f_{q,n}(\C{D}_q)
  = & f_{q,n}\left(\sum_{\sigma\in S_n}\,(-q)^{\ell(\sigma)}\, x_{1\sigma(1)} \ldots x_{n\sigma(n)}\right) \nonumber\\
  = & \sum_{\sigma\in S_n}\,(-q)^{\ell(\sigma)}\,f_{q,n}(x_{1\sigma(1)}) \ldots f_{q,n}(x_{n\sigma(n)})\\
  = & q^{n(n+1)-2(1+\cdots + n)} =  1,\nonumber
\end{align}
and hence $f_{q,n}^{-1}(\C{D}_q^{-1}) = 1$.
\end{proof}

\section{Hochschild Homology of $M_q(n)$}\label{sect:MLN}

\subsection{Homology of quantum matrices $M_q(n,m)$}

Given a sequence $(q_1,\ldots,q_n)$
of scalars in $k$, and let $\alpha \colon M_q(n)\to k$ be the character given by
\[ 
\alpha(x_{ij}) = \delta_{ij} q_i. 
\] 
Accordingly, the counit
\[ \ve(x_{ij}) = \delta_{ij}, \]
the characters $f_{q,n}$ and $f_{q,n}^{-1}$ of Proposition~\ref{MPI-GL-SL}, namely,
\[ f_{q,n}(x_{ij}) = \delta_{ij} q^{2i-1-n} \quad\text{ and }\quad f_{q,n}^{-1}(x_{ij}) = \delta_{ij} q^{n-2i+1} \] and finally the
character $\eta\colon M_q(n,m)\to k$ given by
\begin{equation}\label{trivial-character}
  \eta(x_{ij})=0, 
\end{equation}
for $1\leq i\leq n$ and $1\leq j\leq m$ correspond to the sequences
\begin{align}
  \varepsilon & \leftrightarrow (1,\ldots,1),\\
  f_{q,n} & \leftrightarrow (q^{-n+1},q^{-n+3},\ldots,q^{n-1}),\\
  f_{q,n}^{-1} & \leftrightarrow (q^{n-1},q^{n-3},\ldots,q^{-n+1}),\\
  \eta & \leftrightarrow (0,\ldots,0).
\end{align}
Let, now, $\alpha,\beta\colon M_q(n,m)\to k$ be two characters given by
two sequences of scalars as defined above.  Let also
${}_\alpha k_\beta$ denote the $M_q(n,m)$-bimodule $k$ with the
actions given by
\begin{equation}\label{right-Mqn-act-on-k}
  x_{ij} \tr 1 = \alpha(x_{ij}) \quad\text{ and }\quad 1\tl x_{ij} = \beta(x_{ij})
\end{equation}
for any $x_{ij}\in M_q(n,m)$.  In addition, the absence of a subscript such as
${}_\alpha k$ or $k_\alpha$ indicates that the action on the unspecified side is
given by the counit.

The following result gives us the license to consider only the 1-dimensional
bimodules whose right action is given by the counit.

\begin{proposition}\label{prop:skew}
  Given any characters $\alpha$ and $\beta$ defined by a sequence of scalars
  $(q^{a_1},\ldots,q^{a_n})$ and $(q^{b_1},\ldots,q^{b_n})$ as defined above,
  there is an automorphism $\theta_\alpha\colon M_q(n,m)\to M_q(n,m)$ so that
  the action of $M_q(n,m)$ twisted by $\theta_\alpha$ on ${}_\beta k_\alpha$
  reduces to ${}_{\alpha^{-1}\beta} k$.
\end{proposition}

\begin{proof}
  We define
  \begin{equation}\label{eq:1}
    \theta(x_{ij}) = q^{-a_i} x_{ij}
  \end{equation}
  and observe that the relations~\eqref{quant-mat-I}
  through~\eqref{quant-mat-IV} are invariant under this action.  The action of
  $M_q(n,m)$ twisted by $\theta_\alpha$ is defined as
  \[ 1\blacktriangleleft x_{ij} = \alpha(\theta_\alpha(x_{ij})) =
    \delta_{ij} q^{-a_i} q^{a_i} = \delta_{ij} \]
  and the left action is given as
  \[ x_{ij}\blacktriangleright 1 = \beta(\theta_\alpha(x_{ij}))
     = \delta_{ij} q^{b_i}q^{-a_i} \]
  for every generator
  $x_{ij}$.
\end{proof}

\begin{lemma}\label{lemma-homology-kq-trivial-coeff}
  The Hochschild homology of the quantum affine $n$-space
  $M_q(n,1)\cong M_q(1,n)$ with coefficients in ${}_\eta k_\eta$ is
  given by
  \[ H_\ell(M_q(1,n),\ {}_\eta k_\eta) \cong k^{\oplus\,
      \binom{n}{\ell}}. \]
\end{lemma}

\begin{proof}
  Setting $Q:=k[x_{11}] \subseteq M_q(1,n) =:P$, we have
  \begin{align*}
    H_\ast(M_q(1,n),\ {}_\eta k_\eta)
    \Leftarrow E^1_{i,j}
    = & H_j(k[x_{11}],\ {}_\eta k_\eta \ot_{k[x_{11}]} \, \underbrace{M_q(1,n) \ot_{k[x_{11}]} \ldots \ot_{k[x_{11}]} \, M_q(1,n)}_\text{$i$-many}) \\
    \cong & H_j(k[x_{11}],\ {}_\eta k_\eta \ot \, \underbrace{M_q(1,n-1) \odots \, M_q(1,n-1)}_\text{$i$-many}),
  \end{align*}
  where the $k[x_{11}]$ action is still given by $\eta$ on the coefficient complex. Thus,
  the $E^1$-term of the spectral sequence splits as
  \begin{align*}
    E^1_{i,j}
    = & H_j(k[x_{11}],\ {}_\eta k_\eta)\ot \, \underbrace{M_q(1,n-1) \odots \, M_q(1,n-1)}_\text{$i$-many}\\
    \cong & \CH_i(M_q(1,n-1),\ {}_\eta k_\eta)\otimes H_j(k[x_{11}],\ {}_\eta k_\eta)
  \end{align*}
  since the action of $M_q(1,n-1)$ on $H_j(k[x_{11}],k)$ is again given by $\eta$.
  On the other hand for $k[x_{11}]$ we have
  \[ H_j(k[x_{11}],\ {}_\eta k_\eta) =
    \begin{cases}
      k & \text{ if } j=0,1\\
      0 & \text{ otherwise.}
    \end{cases}
  \] 
  Then we see that 
  \[ H_\ell(M_q(1,n),\ {}_\eta k_\eta)
    \cong H_\ell(M_q(1,n-1),\ {}_\eta k_\eta)\oplus
          H_{\ell-1}(M_q(1,n-1),\ {}_\eta k_\eta). \]
  The result follows from recursion.
\end{proof}

Let $\Lambda^*(X)$ denote the exterior algebra generated by a set $X$ of
indeterminates.  The following result follows from an easy dimension
counting.
\begin{proposition}\label{prop:quantum-matrices-with-trivial-coeff}
  We have isomorphisms of vector spaces of the form
  \[ H_\ell(M_q(n,m),{}_\eta k_\eta) \cong \Lambda^\ell(x_{ij}\mid
    1\leq i\leq n,\ 1\leq j\leq m) \] for every $m,n\geq 1$ and
  $\ell\geq 0$.
\end{proposition}

\begin{proof}
  Let prove this by induction on $n$. For $n=1$ the result is given by
  Lemma~\ref{lemma-homology-kq-trivial-coeff}.  Assume we have the
  prescribed result for $n$.  Consider the extension
  $M_q(n,m)\subseteq M_q(n+1,m)$ with the canonical embedding.  Then
  by Proposition~\ref{prop:HochschildSS} we get
  \begin{align*}
    H_*(M_q(n+1,m),\ {}_\eta k_\eta)
    \Leftarrow E^1_{i,j} = & H_j(M_q(n,m), \underbrace{M_q(n+1,m)\otimes_{M_q(n,m)}\cdots\otimes_{M_q(n,m)} M_q(n+1,m)}_\text{$i$-times})\\
    \cong & H_j(M_q(n,m), \CH_i(M_q(1,m),\ {}_\eta k_\eta))\\
    \cong & H_j(M_q(n,m),\ {}_\eta k_\eta) \otimes \CH_i(M_q(1,m),\ {}_\eta k_\eta)
  \end{align*}
  since $M_q(n,m)$ acts by $\eta$ on the coefficient complex.  Accordingly,
  \begin{align*}
    H_\ell(M_q(n+1,m),{}_\eta k_\eta)
    \cong \bigoplus_{i+j=\ell} H_j(M_q(n,m),\ {}_\eta k_\eta) \otimes H_i(M_q(1,m),\ {}_\eta k_\eta),
  \end{align*}
  and therefore
  \begin{align*}
  \dim_k H_\ell(M_q(n+1,m),\ {}_\eta k_\eta)
    = & \sum_{\ell_1+\ell_2 = \ell} \dim_k H_{\ell_1}(M_q(n,m),\ {}_\eta k_\eta)
        \cdot\dim_k H_{\ell_2}(M_q(1,m),\ {}_\eta k_\eta)\\
    = & \sum_{\ell_1+\ell_2 = \ell} \binom{nm}{\ell_1}\binom{m}{\ell_2} = \binom{(n+1)m}{\ell}
  \end{align*}
  as we wanted to show.
\end{proof}

\subsection{Homology of $M_q(n)$}

In this subsection we compute the Hochschild homology of the algebra
$M_q(n)$ of quantum matrices with coefficients in
$\alpha = (q^{a_1},\ldots,q^{a_n})$, where $a_1,\ldots,a_n\in\B{Z}$.

To this end, we begin with the extension $\text{Row}_n\subseteq M_q(n)$ that yields, in the relative complex,
\begin{align*}
  {}_\alpha k \otimes_{\text{Row}_n} & \underbrace{M_q(n)\otimes_{\text{Row}_n}\cdots \otimes_{\text{Row}_n} M_q(n)}_{\text{$i$-times}}\\
  \cong & {}_\alpha k \otimes \underbrace{M_q(n-1,n)\otimes\cdots \otimes M_q(n-1,n)}_{\text{$i$-times}}.
\end{align*}
Hence, the $E^1$-page of the spectral sequence is
\[ 
E^1_{i,j} \cong H_j(\text{Row}_n,\CH_i(M_q(n-1,n),\ {}_\alpha k)). 
\]
We then note that the elements $x_{ni}\in \text{Row}_n$, for $i\neq n$, act on
the coefficient complex via $\eta$, whereas $x_{nn}$ act via a scalar $q^{a_n}$
on the left. On the right, the action of $x_{nn}$ is via another scalar
determined by the total degree of the terms in $x_{in}$ in
$\CH_i(M_q(n-1,n),{}_\alpha k)$ for $1\leq i\leq n-1$.  Accordingly,
\[ E^1_{i,j} \cong \bigoplus_a H_j(\text{Row}_n,\
  {}_{q^{a_n}}k_{q^{-a}}) \otimes \CH_i^{(a)}(M_q(n-1,n),\ {}_\alpha
  k),\] where $\CH_*^{(a)}$ denotes the subcomplex of
terms whose total degree in $x_{in}$, for $i=1,\ldots,n-1$, are precisely
$a \in \B{Z}$.  Let us remark also that since these terms act by $\eta$, the graded subspace
$\CH_*^{(a)}(M_q(n-1,n),\ {}_\alpha k)$ of
$\CH_*(M_q(n-1,n),\ {}_\alpha k)$ is indeed a subcomplex.  

For the homology of the row algebra this time, we use the lattice of
extensions $\text{Row}_n(a,b)\subseteq \text{Row}_n$, where
$\text{Row}_n(a,b)$ is the subalgebra of $\text{Row}_n$ generated by
$x_{na},\ldots,x_{nb}$. Thus, we may express
\[ H_j (\text{Row}_n, {}_{q^{a_n}}k_{q^{-a}})
  = \bigoplus_c H_c(\text{Row}_n(b+1,n), {}_{q^{a_n}}k_{q^{-a+c-j}})
    \otimes\Lambda{}^{j-c}(x_{n1},\ldots,x_{nb})
\]
for every $1\leq b\leq n-1$.  In particular,
\[ H_j (\text{Row}_n,{}_{q^{a_n}}k_{q^{-a}})
  = \bigoplus_c H_c(\text{Row}_n(n,n), {}_{q^{a_n}}k_{q^{-a+c-j}})
    \otimes\Lambda^{j-c}(x_{n1},\ldots,x_{nn-1}).
\]
Since $\text{Row}_n(n,n) = k[x_{nn}]$, the direct sum above has
only two non-zero terms: those with $c=0$ and $c=1$. Therefore,
\begin{align*}
  H_j (\text{Row}_n,{}_{q^{a_n}}k_{q^{-a}})
  = & \Big(H_0(k[x_{nn}], {}_{q^{a_n}}k_{q^{-j-a}})\otimes\Lambda^j (x_{n1},\ldots,x_{nn-1})\Big)\\
    & \oplus \Big(H_1(k[x_{nn}], {}_{q^{a_n}}k_{q^{-j-a+1}})\otimes\Lambda^{j -1}(x_{n1},\ldots,x_{nn-1})\Big).
\end{align*}
On the other hand, we observe that the homology is zero unless ${}_{q^{a_n}}k_{q^{-a+c-j}}$ is symmetric.  Thus,
\begin{align*}
  H_j(\text{Row}_n,{}_{q^{a_n}}k_{q^{-a}})= & \begin{cases}
    \Lambda^j _q(x_{n1},\dots,x_{nn-1}) & \text{ if } a =  -a_n-j\geq 0,\\
    \Lambda^{j-1}(x_{n1},\dots,x_{nn-1})x_{nn} & \text{ if } a = -a_n-j+1\geq 0,\\
    0 & \text{ otherwise.}
  \end{cases}
\end{align*}
As a result, the $E^1$-page of the spectral sequence reduces to
\begin{align*}
  H_\ell(M_q(n), \,& {}_{\alpha}k)\\
  \Leftarrow  E^1_{i,j}\cong
  & \CH_i^{(-a_n-j)}(M_q(n-1,n),\ {}_\alpha k\otimes \Lambda^j (x_{n1},\dots,x_{n,n-1})) \\
  & \oplus
    \CH_i^{(-a_n-j+1)}(M_q(n-1,n),\ {}_\alpha k\otimes \Lambda^{j-1}(x_{n1},\dots,x_{n,n-1})x_{nn}).
\end{align*}
Now, let $\text{Col}_n(a,b)$ be the subalgebra generated by
$x_{an},\ldots,x_{bn}$ in $M_q(n-1,n)$. Then, in view of Proposition~\ref{prop:HochschildSS} we have
\begin{align*}
  H_i^{(a)}(M_q(n-1,n),\, & {}_\alpha k\otimes \Lambda^j(x_{1n},\ldots,x_{nn}))\\
  & \Leftarrow E^1_{r,s} = H_s^{(a)}(\text{Col}_n(1,n-1),\CH_r(M_q(n-1),{}_\alpha k\otimes \Lambda^j(x_{1n},\ldots,x_{nn}))).
\end{align*}
Since $\text{Col}_n(1,n-1)$ acts on the coefficient complex via $\eta$,
 the $E^1$-page splits, and we arrive at
\begin{align*}
  H_i^{(a)}(M_q(n-1,n), \,& {}_\alpha k\otimes \Lambda^j(x_{1n},\ldots,x_{nn}))\\
  \cong & H_{i-a}(M_q(n-1),{}_\alpha k\otimes \Lambda^j(x_{1n},\ldots,x_{nn})\otimes \Lambda^a(x_{1n},\ldots,x_{n-1,n})).
\end{align*}
From these we get the $E^2$-page
\begin{align*}
E^2_{i,j}
  \cong & 
    H_{i+j+a_n}(M_q(n-1),\ {}_\alpha k\otimes \Lambda^j (x_{n1},\dots,x_{n,n-1})\otimes \Lambda^{-a_n-j}(x_{1n},\ldots,x_{n-1,n})) \\
    & \oplus H_{i+j+a_n-1}(M_q(n-1),\ {}_\alpha k\otimes \Lambda^{j-1}(x_{n1},\dots,x_{n,n-1})x_{nn}\otimes \Lambda^{-a_n-j+1}(x_{1n},\ldots,x_{n-1,n})).
\end{align*}
Therefore,
\begin{align*}
 H_\ell( & M_q(n), {}_\alpha k)\\
  \cong & \bigoplus_{j,s}
          H_{\ell+a_n-s}(M_q(n-1),\ {}_\alpha k \otimes \Lambda^j(x_{n1},\ldots,x_{n,n-1})
          \otimes \Lambda^{-a_n-j}(x_{1n},\ldots,x_{n-1,n}))\otimes\Lambda^s(x_{nn}).
\end{align*}

Now, consider the subspace $S_n^*$ of $\Lambda^*(x_{ij}\mid 1\leq i,j\leq n)$
generated by $x_{in}$ and $x_{nj}$ with $1\leq i\leq n-1$ and $1\leq j\leq n-1$.
Also, we use $S_n^*(b)$ to denote the homogeneous vector subspace of $S_n^*$ of
terms whose total degree over terms of type $x_{in}$ and $x_{ni}$ is $b$.

We observe that
\begin{align*}
  H_\ell(M_q(n),{}_\alpha k)
  \cong & \bigoplus_s H_{\ell+a_n-s}(M_q(n-1),\ {}_\alpha k \otimes S_{n-1}^*(-a_n))\otimes\Lambda^s(x_{nn})\\
  \cong & \bigoplus_{\beta,s} H_{\ell+a_n-s}(M_q(n-1),\ {}_{\alpha} k_{\beta^{-1}}) \otimes S_{n-1}^*(b_1,\ldots,b_{n-1},-a_n)\otimes\Lambda^s(x_{nn})\\
  \cong & \bigoplus_{\beta,s} H_{\ell+a_n-s}(M_q(n-1),\ {}_{\alpha\beta} k) \otimes S_{n-1}^*(b_1,\ldots,b_{n-1},-a_n)\otimes\Lambda^s(x_{nn}).  
\end{align*}
The sum is taken over all $\beta=(q^{b_1},\ldots,q^{b_n})$, where the multi-degree $(b_1,\ldots,b_n)$
indicates that we consider the $k$-vector space spanned by monomials
$\Gamma$ that has the total degree $\deg_i(\Gamma) = b_i$ in terms of
$x_{si}$ and $x_{it}$ for all $s,t=1,\ldots,n$ and $i=1,\ldots,n$, and
we set
\begin{equation}\label{eq:degree}
  \deg_i(x_{ij}) = \deg_i(x_{ji}) = 
    \begin{cases}
      -1 & \text{ if } j<i,\\
       0 & \text{ if } j=i,\\      
       1 & \text{ if } j>i.
    \end{cases}
\end{equation}
It follows at once that $b_n = -a_n$, since there are no indices $j>n$.  Proceeding the computation recursively, we arrive at the following result.

\begin{theorem}
  Fix a sequence of non-zero scalars
  $\alpha = (q^{a_1},\ldots,q^{a_n})$, and let us define
  \[ \Lambda^*_{(\alpha)}(x_{ij}\mid 1\leq i,j\leq n) \] as the subspace of
  differential forms with multi-degree $(a_1,\ldots,a_n)$ where the $a_i$ is the
  total degree -- in the sense of \eqref{eq:degree} -- of terms involving the
  indeterminates $x_{si}$ and $x_{it}$ for $s,t=1,\ldots,n$, and $i=1,\ldots,n$.
  Then
  \[ H_{\ell+|\alpha|}(M_q(n),{}_\alpha k) \cong \bigoplus_s
    \Lambda^{\ell-s}_{(\alpha)}(x_{ij}\mid 1\leq i\neq j\leq n)\otimes\Lambda^s(x_{11},\ldots,x_{nn})
  \] as vector spaces for every $n\geq 1$, $\ell\geq 0$.  
\end{theorem}

\subsection{Homologies of $GL_q(n)$ and $SL_q(n)$ with coefficients in
  ${}_{f_{q,n}^{-1}} k$}

We are going to derive the homology of $GL_q(n)$ and $SL_q(n)$ from that of $M_q(n)$ by using the localization of the Hochschild homology.

Let us first recall the localization of the homology, \cite[Prop. 1.1.17]{Loday-book}.  
\begin{proposition}\label{prop-local}
Given an algebra $A$, and a multiplicative subset $S \subseteq A$ so that $1\in S$ and $0 \notin S$, and an $A$-bimodule $M$, there are the following canonical isomorphisms:
\[
H_\ast(A,M)_S \cong H_\ast(A,M_S) \cong H_\ast(A_S,M_S),
\]
where
\[
M_S := Z(A)_S \ot_{Z(A)} M,
\]
$Z(A)$ denotes the center of $A$, and $Z(A)_S$ stands for the localization of $Z(A)$ at $S$.
\end{proposition}

Now, in view of the fact that $GL_q(n)$ is the localization of $M_q(n)$ at $S=\{\C{D}_q^n\mid n\geq 0\}$,
we obtain the Hochschild homology of $GL_q(n)$ readily from the above localization result. 
\begin{theorem}\label{thm:GLN-SLN}
  We have
  \[ H_\ell(M_q(n),{}_{f_{q,n}^{-1}}k) = H_\ell(GL_q(n),{}_{f_{q,n}^{-1}}k) 
    \cong H_\ell(SL_q(n),{}_{f_{q,n}^{-1}}k) \oplus
    H_{\ell-1}(SL_q(n),{}_{f_{q,n}^{-1}}k)
  \]
  for every $\ell\geq 0$ and for every $n\geq 1$.
\end{theorem}

\begin{proof}
Let us recall from \cite[Thm. 1.6]{NoumYamaMima93}, see also \cite{FRT89}, that 
\[
Z(M_q(n)) = k[\C{D}_q],
\]
and that $GL_q(n) = M_q(n)_S$ for the multiplicative system $S=\{\C{D}_q^n\mid n\geq 0\}$ generated by the quantum determinant.
Accordingly, we have
\[
Z(M_q(n))_S = k[\C{D}_q, \C{D}_q^{-1}],
\]
and
\[
{}_{f_{q,n}^{-1}}k_S = k[\C{D}_q^{-1}] \ot {}_{f_{q,n}^{-1}}k
\]
is the $GL_q(n)$-bimodule so that the $M_q(n)$-bimodule structure concentrated on ${}_{f_{q,n}^{-1}}k$, and the $k[\C{D}_q^{-1}]$-bimodule structure is on $k[\C{D}_q^{-1}]$. Then we have
\[
H_\ast(GL_q(n),{}_{f_{q,n}^{-1}}k_S) \cong k[\C{D}_q^{-1}] \ot H_\ast(GL_q(n),{}_{f_{q,n}^{-1}}k)
\]
so that the $GL_q(n)$-bimodule structure on ${}_{f_{q,n}^{-1}}k$ is determined by the trivial action of $\C{D}_q^{-1}$.  On the other hand,
\[
H_\ast(M_q(n),{}_{f_{q,n}^{-1}}k)_S \cong k[\C{D}_q^{-1}] \ot H_\ast(M_q(n),{}_{f_{q,n}^{-1}}k),
\]
where the $GL_q(n)$-bimodule structure is given in such a way that the
$M_q(n)$-bimodule structure is on $H_\ast(M_q(n),{}_{f_{q,n}^{-1}}k)$, and
the $\C{D}_q^{-1}$-action structure is concentrated on
$k[\C{D}_q^{-1}]$.  Proposition \ref{prop-local} then yields the first
claim.  The second part the computation follows
from~\cite[Proposition]{LevaStaff93} and \cite[Theorem
2.7]{KaygSutl18}.
\end{proof}

\subsection{Homologies of $GL_q(n)$ and $SL_q(n)$ with coefficients in themselves}

Let $H$ be a Hopf algebra with an invertible antipode, and let $X$ be
an arbitrary $H$-bimodule.
\begin{proposition}\label{prop:maclane-isomorphism}
  There is an isomorphism of graded vector spaces
  $H_*(H,X)\cong \tor^H_*(ad(X),k)$ where we define the adjoint action
  of $H$ on $X$ as
  \[ x^h := S^{-1}(h_{(1)})xh_{(2)} \] for every $x\in X$ and
  $h\in H$.
\end{proposition}

\begin{proof}
  Let $\CB_*(H)$ be the bar complex of $H$ viewed as an algebra and we
  will use $\CB_*(X,H,Y)$ to denote $X\otimes_H \CB_*(H)\otimes_H Y$
  to denote the two sided complex with coefficients in a left
  $H$-module $Y$ and right $H$-module $X$.  Let us write an
  isomorphism of complexes
  $\rho_*\colon \CH_*(H,X)\to \CB_*(ad(X),H,k)$ as
  \[ \rho_n(x\otimes h_1\otimes\cdots\otimes h_n) = h_{1,(1)}\cdots
    h_{n,(1)}x\otimes h_{1,(2)}\otimes\cdots\otimes h_{n-1,(2)}\otimes
    h_{n,(2)} \] It is straight-forward but tedious exercise that
  $\rho_*$ is an isomorphism of pre-simplicial $k$-modules, and
  therefore, an isomorphism of complexes.
\end{proof}

The isomorphism given in Proposition~\ref{prop:maclane-isomorphism}
is a well-known isomorphism used in (co)homology of Hopf algebras,
and usually referred as \emph{MacLane Isomorphism}.  This is an
extension of the same result for a group algebras, and universal
enveloping algebras.

Let us use ${}_{f_{n,q}^{-1}}H$ to denote the regular representation of
$H=SL_q(n)$ or $H=GL_q(n)$ twisted on the left by the automorphism
$\theta_{f_{n,q}^{-1}}$ defined in Proposition~\ref{prop:skew}.
\begin{theorem}
  The canonical map
  $\C{D}_q\colon {}_{f_{q,n}^{-1}} k\to ad({}_{f_{q,n}^{-1}} GL_q(n))$
  induces a split injection in homology of the form
  $H_m(GL_q(n),{}_{f_{q,n}^{-1}} k)\to
  H_m(GL_q(n),{}_{f_{n,q}^{-1}}GL_q(n))\cong H_{m+1}(GL_q(n))$ and
  $H_m(SL_q(n),{}_{f_{q,n}^{-1}} k)\to
  H_m(SL_q(n),{}_{f_{q,n}^{-1}}SL_q(n))\cong H_{m+1}(SL_q(n))$ for
  every $m\geq 0$.
\end{theorem}

\begin{proof}
  We use Proposition~\ref{prop:maclane-isomorphism} for $H=GL_q(n)$ with
  coefficients in the twisted module ${}_{f_{q,n}^{-1}}H$.  The adjoint module
  $ad({}_{f_{q,n}^{-1}}H)$ splits as a direct sum of irreducible submodules with
  multiplicities, and one of those modules is the module ${}_{f_{q,n}^{-1}}k$.
  In the case $H=GL_q(n)$ or $H=SL_q(n)$, the quantum determinant $\C{D}_q$ lies
  inside that submodule since $\C{D}_q$ is in the center.  Then
  $\tor^H_*({}_{f_{q,n}^{-1}} k,k)$ is a direct summand of
  $H_*(H,{}_{f_{q,n}^{-1}}H)$.  The rest follows from an untwist as
  in~\cite[Section 2.9]{KaygSutl16} but for Hochschild homology.
\end{proof}

\section{Explicit calculations}\label{sect:calculations}

% \subsection{$M_q(1)$, $GL_q(1)$ and $SL_q(1)$}
% The character $f_{q,1}^{-1}$ is given by the sequence $(1)$, and
% therefore $|\alpha|=0$, and $H_\ell(M_q(1),k) = \Lambda^\ell(x_{11})$.

\subsection{$M_q(2)$, $GL_q(2)$ and $SL_q(2)$}

The character $f_{q,2}^{-1}$ is given by the sequence $(q,q^{-1})$.
Then, $|\alpha| = 0$ and
\[ H_\ell(M_q(2),{}_{(q,q^{-1})}k) = \bigoplus_s \Lambda^{\ell-s}_{(q,q^{-1})}(x_{12},x_{21})
  \otimes \Lambda^s(x_{11},x_{22}). \]
One can write 4 different exterior product between $x_{12}$ and $x_{21}$, and
\begin{align}
  \deg((1)) = & (0,0)\nonumber\\
  \deg((x_{21})) = \deg((x_{12})) = & (1,-1)\\
  \deg((x_{12},x_{21})) = & (2,-2)\nonumber
\end{align}
from which we only take the degree $(1,-1)$-terms of exterior degree 1.  Thus
\begin{equation}
  H_\ell(M_q(2),{}_{(q,q^{-1})}k) =
  \begin{cases}
    0 & \text{ if $\ell=0$ or $\ell\geq 4$},\\
    Span_k((x_{12}),(x_{21})) & \text{ if } \ell=1,\\
    Span_k((x_{11},x_{12}),(x_{11},x_{21}),(x_{12},x_{22}),(x_{21},x_{22})) & \text{ if }\ell=2,\\
    Span_k((x_{11},x_{12},x_{22}),(x_{11},x_{21},x_{22})) & \text{ if } \ell=3.
  \end{cases}  
\end{equation}
The Betti numbers of the homology are given in Figure~\ref{fig:1}.

\begin{figure}[ht]
  \centering
  \begin{tabular}[]{|r|c|c|c|}\hline
    $m$ & 1 & 2 & 3 \\\hline
    $\dim_k H_m(M_q(2),{}_{f_{q,2}^{-1}} k)$ & 2 & 4 & 2 \\\hline
    $\dim_k H_m(GL_q(2),{}_{f_{q,2}^{-1}} k)$ & 2 & 4 & 2 \\\hline
    $\dim_k H_m(SL_q(2),{}_{f_{q,2}^{-1}} k)$ & 2 & 2 &  \\\hline
  \end{tabular}  
  \caption{The Betti numbers for $M_q(2)$, $GL_q(2)$ and $SL_q(2)$}
  \label{fig:1}
\end{figure}

\subsection{$M_q(3)$, $GL_q(3)$ and $SL_q(3)$}

The character $f_{q,3}^{-1}$ is given by
$(q^2,1,q^{-2})$ and $|\alpha|=0$.  The exterior degree 1 terms are
\begin{align}
  \deg((x_{12})) = \deg((x_{21})) = & (1,-1,0) \nonumber\\
  \deg((x_{13})) = \deg((x_{31})) = & (1,0,-1) \\
  \deg((x_{23})) = \deg((x_{32})) = & (0,1,-1) \nonumber
\end{align}
and we need terms of degree signature $(2,0,-2)$.  We must solve a
system of $\B{Z}$-linear equations
\[(\alpha_1,\alpha_2,\alpha_3)
  \left[\begin{matrix}
      1 & -1 & 0\\
      1 &  0 & -1\\
      0 &  1 & -1
    \end{matrix}\right]
  = (2,0,-2),
\]
where $\alpha_i\in\{0,1,2\}$.  The only solutions are
\[ (\alpha_1,\alpha_2,\alpha_3) = (2,0,2) \quad\text{ or }\quad
  (\alpha_1,\alpha_2,\alpha_3) = (1,1,1).\] For the first solution,
there is only one term of exterior degree 4:
$(x_{12},x_{21},x_{23},x_{32})$.  On the other hand, for the second
solution there are 8 such terms of exterior degree 3.  Then we use the
exterior algebra on $x_{11}$, $x_{22}$ and $x_{33}$ to promote these
terms to higher degrees.  In short, we have:
\begin{align}
  H_{3+\ell}(M_q(3), & {}_{(q^2,1,q^{-2})}k) \nonumber\\
  = Span_k\Big( & (x_{12},x_{13},x_{23}),
                  (x_{12},x_{13},x_{32}),
                  (x_{12},x_{31},x_{23}),
                  (x_{12},x_{31},x_{32}),\nonumber\\
              & \qquad (x_{12},x_{13},x_{23}),
                (x_{12},x_{13},x_{32}),
                (x_{12},x_{31},x_{23}),
                (x_{12},x_{31},x_{32})\Big)\otimes\Lambda^\ell(x_{11},x_{22},x_{33})\\
  & \oplus Span_k((x_{12},x_{21},x_{23},x_{32}))\otimes\Lambda^{\ell-1}(x_{11},x_{22},x_{33}).\nonumber
\end{align}
The Betti numbers of the homology are given in Figure~\ref{fig:2}.

\begin{figure}[ht]
  \centering
  \begin{tabular}[]{|r|c|c|c|c|c|}\hline
    $m$ & 3 & 4 & 5 & 6 & 7 \\\hline
    % (a) & 8 & 24 & 24 & 8 & 0\\\hline
    % (b) & 0 & 1 & 3 & 3 & 1\\\hline
    $\dim_k H_m(M_q(3),{}_{f_{q,3}^{-1}} k) $ & 8 & 25 & 27 & 11 & 1\\\hline
    $\dim_k H_m(GL_q(3),{}_{f_{q,3}^{-1}} k) $ & 8 & 25 & 27 & 11 & 1\\\hline
    $\dim_k H_m(SL_q(3),{}_{f_{q,3}^{-1}} k) $ & 8 & 17 & 10 & 1 & \\\hline
  \end{tabular}  
    \caption{The Betti numbers for $M_q(3)$, $GL_q(3)$ and $SL_q(3)$}
  \label{fig:2}
\end{figure}

\subsection{$M_q(4)$, $GL_q(4)$ and $SL_q(4)$}

The character $f_{q,4}^{-1}$ is now given by the sequence
$\alpha = (q^3,q,q^{-1},q^{-3})$ with $|\alpha|=0$.  The exterior
degree 1 terms are
\begin{align}
  \deg((x_{12})) = \deg((x_{21})) = & (1,-1,0,0)
  & \deg((x_{23})) = \deg((x_{32})) = & (0,1,-1,0) \nonumber\\
  \deg((x_{13})) = \deg((x_{31})) = & (1,0,-1,0)
  & \deg((x_{24})) = \deg((x_{42})) = & (0,1,0,-1)\\
  \deg((x_{14})) = \deg((x_{41})) = & (1,0,0,-1)  \nonumber
  & \deg((x_{34})) = \deg((x_{43})) = & (0,0,1,-1)
\end{align}
and we need the total multi-degree $(3,1,-1,-3)$.  Thus we solve
\[(\alpha_1,\alpha_2,\alpha_3,\alpha_4,\alpha_5,\alpha_6)
  \left[\begin{matrix}
      1 & -1 &  0 &  0\\
      1 &  0 & -1 &  0\\
      1 &  0 &  0 & -1\\ 
      0 &  1 & -1 &  0\\
      0 &  1 &  0 & -1\\
      0 &  0 &  1 & -1
    \end{matrix}\right]
  = (3,1,-1,-3)
\]
again with the restriction that $\alpha_i\in\{0,1,2\}$.
The Betti numbers of this case are given in Figure~\ref{fig:3}.

\begin{figure}[ht]
  \centering
  \begin{tabular}[]{|r|c|c|c|c|c|c|c|c|c|c|c|c|c|}\hline
     $m$  & 2 & 3 & 4 & 5 & 6  &  7 & 8 & 9 & 10 & 11 & 12 & 13 & 14 \\\hline
    %  (a)  &   &   &   &   & 64 &256 &384&256& 64 &    &    &    &  \\\hline
    %  (b)  &   & 8 & 32& 48& 32 &  8 &   & 8 & 32 & 48 & 32 & 8  & \\\hline
    % (c1)  &   &   &   & 8 & 32 & 48 & 32& 8 &    &    &    &    & \\\hline
    % (c2)  &   &   &   &   &    &  8 & 32& 48& 32 & 8  &    &    & \\\hline
    % (d1)  &   &   &   & 8 & 32 & 48 & 32& 8 &    &    &    &    & \\\hline
    % (d2)  &   &   &   &   &    &  8 & 32& 48& 32 & 8  &    &    & \\\hline
    % (e1)  & 4 & 16& 24& 16& 4  &    &   &   &    &    &    &    & \\\hline
    % (e2)  &   &   &   &   &    &    &   &   &  4 & 16 & 24 & 16 & 4\\\hline
    % (f1)  & 4 & 16& 24& 16& 4  &    &   &   &    &    &    &    & \\\hline
    % (f2)  &   &   &   &   &    &    &   &   &  4 & 16 & 24 & 16 & 4\\\hline
    % (g)   &   &   &   &   & 8  & 32 & 48& 32& 8  &    &    &    & \\\hline
    $\dim_k H_m(M_q(4),{}_{f_{q,4}^{-1}} k)$
         & 8& 40 & 80 & 96& 176& 408& 560& 408& 176& 96& 80& 40 & 8\\\hline
    $\dim_k H_m(GL_q(4),{}_{f_{q,4}^{-1}} k)$
         & 8& 40 & 80 & 96& 176& 408& 560& 408& 176& 96& 80& 40 & 8\\\hline
    $\dim_k H_m(SL_q(4),{}_{f_{q,4}^{-1}} k)$
         & 8 & 32 & 48 & 48& 128& 280& 280& 128& 48& 48 & 32 & 8 & \\\hline
  \end{tabular}  
 \caption{The Betti numbers for $M_q(4)$, $GL_q(4)$ and $SL_q(4)$}
  \label{fig:3}
\end{figure}

\bibliographystyle{plain}
\bibliography{references}{}

\end{document}

%% file: preamble.tex
\usepackage[a4paper,margin=25mm,top=31mm]{geometry}

\usepackage{amsfonts, amsmath, amssymb, amsgen, amsthm, amscd}
\usepackage{newtxtext,newtxmath}
\usepackage[utf8]{inputenc} 

\usepackage{color}

\usepackage[all]{xy}

\usepackage{hyperref}
\usepackage{mathtools,slashed}

\usepackage{enumerate}

\def\d{\delta}
\def\D{\Delta}

\def\s{\sigma}

\def\ve{\varepsilon}
\def\vp{\varphi}

\def\Id{\mathop{\rm Id}\nolimits}

\def\lra{\longrightarrow}

\def\ot{\otimes}
\def\odots{\ot\cdots\ot}

\def\lra{\longrightarrow}

\def\rt{\triangleright}
\def\lt{\triangleleft}

\def\lbiprod{{>\!\!\!\triangleleft\kern-.33em\cdot\, }}
\def\rbiprod{{\cdot\kern-.33em\triangleright\!\!\!<}}

\def\D{\Delta}

\def\Id{\mathop{\rm Id}\nolimits}

\usepackage{enumerate}

\newcommand{\G}[1]{\mathfrak{#1}}
\newcommand{\C}[1]{\mathcal{#1}}
\newcommand{\B}[1]{\mathbb{#1}}

\newcommand{\CB}{{\rm CB}}
\newcommand{\CH}{{\rm CH}}

\newcommand{\Rw}{\text{Row}}
\newcommand{\Cl}{\text{Col}}

\newcommand{\tor}{{\rm Tor}}

\newcommand{\tr}{\triangleright}
\newcommand{\tl}{\triangleleft}

\renewcommand{\leq}{\leqslant}
\renewcommand{\geq}{\geqslant}

\numberwithin{equation}{section}

\newtheorem{theorem}{Theorem}[section]
\newtheorem{proposition}[theorem]{Proposition}
\newtheorem{lemma}[theorem]{Lemma}

\theoremstyle{definition}